\documentclass[a4paper,12pt]{article} 
\usepackage{amsmath, amsthm, amssymb}
\usepackage{url}

\usepackage{physics}
\usepackage{graphicx,lipsum}
\graphicspath{ {./downloads/} }

\usepackage[colorlinks,citecolor=red,urlcolor=blue,bookmarks=false,hypertexnames=true]{hyperref}
\usepackage{tikz}
\usetikzlibrary{calc}
\usetikzlibrary{shapes}
\usepackage[autostyle]{csquotes}
\makeatletter

\usepackage[colorlinks]{hyperref}
\usepackage[nameinlink,capitalize]{cleveref}
\theoremstyle{plain}
\newtheorem{thm}{Theorem}[section]
\newtheorem{cor}[thm]{Corollary}
\newtheorem{lem}[thm]{Lemma}
\newtheorem{prop}[thm]{Proposition}
\newtheorem{defn}[thm]{Definition}
\newtheorem{exa}[thm]{Example}

\usepackage{xspace}
\usepackage[margin=1.0in]{geometry}

\usepackage{titlesec}

\usepackage{mathtools}

\DeclarePairedDelimiterX{\inp}[2]{\langle}{\rangle}{#1, #2}
\titleformat{\chapter}
  {\Large\bfseries} 
  {}                
  {0pt}            
  {\huge}

\begin{document}

\begin{center}
\fontsize{13pt}{10pt}\selectfont
    \textsc{\textbf{ Graded Prime Ideals Over Graded Near Ring}}
    \end{center}
\vspace{0.1cm}
\begin{center}
   \fontsize{12pt}{10pt}\selectfont
    \textsc{{\footnotesize    MaliK Bataineh, Tamem Al-shorman and Eman Al-Kilany  }}
\end{center}
\vspace{0.2cm}

\begin{abstract}
  In this paper, we consider graded near-rings over a monoid
$G$ as a generalizations of graded rings over groups. We introduce
certain innovative graded prime ideals and study some of its basic
properties over graded near-rings.  
\end{abstract}

\section{INTRODUCTION}
Near-rings are generalizations of rings: addition is not necessarily
abelian and only one distributive law holds. They arise in a natural
way in the study of mappings on groups:  the set $M(G)$ of all maps
of a group $(G,+)$ into itself endowed with pointwise addition and
composition of functions is a near-ring. For general background on
the theory of near-rings we refer the reader to the monographs
written by Pilz [5] and Meldrum [3]. In fact Pilz  defines it as : A
near-ring $(N,+,*)$ is a set $N$ with two
binary operations $+$ and $*$ that satisfy the following axioms:\\
$(1)$ $(N,+)$ is a group.\\
$(2)$ $(N,*)$ is semi group. (semi group: a set together with an
associative binary operation);\\
$(3)$ $*$ is right distributive over $+$ (i.e. $(a+b)*y=ay+by)$. \\
The graded rings were introduced by Nastasescu and Van Oystaeyen
[4]. Also graded near-rings were introduced and studied by Dumitru,
Nastasescu and Toader [1] in which they defined it as:\\
Let G be a multiplicatively monoid (an algebraic structure with a
single associative
 binary operation) with identity, a near-ring N is called a G-graded if there exists a family of
 additive normal subgroups $\{ N_{\sigma}\}$  of N satisfying that:
\begin{enumerate}
    \item  $ N =\ \oplus_ {\sigma \in G} N_{\sigma}  $ ; \item $N_{\sigma}N_{\tau} \subseteq N_{\sigma \tau} , \forall \sigma , \tau \in G.$
\end{enumerate}.
Graded prime ideals over graded rings have been introduced and
defined as:\\
Let $R$ be a commutative $G$-graded ring where $G$ is an abelian
group,
       a graded proper ideal $I$ of $R$ is called  graded prime ideal if
       $ a_gb_h  \in  I_{gh} \ $, then  $a_g \in I_g \ or\ b_h \in
       I_h$. (see [2]).\\ 

\section{GRADED PRIME IDEALS }
In this Section, we present the concept of graded prime ideals,
study some of its properties in graded near rings and in special
graded near-rings such as graded near-rings with zero ideal is
graded prime and graded near-rings in which every
   graded ideal is graded prime. Also, determine the shape of graded prime ideal in the graded quotient near-ring and in the
     product of two $G$-graded near-rings.
\begin{defn} Let G be a multiplicatively monoid with identity element and
$N$
  be a $G$-graded near-ring.  A proper graded  ideal $P$ is called graded  prime
  ideal if whenever $ A_gB_h \subseteq P_{gh} \ $, then \ either\ $A_g \subseteq  P_g\ $ or $\ B_h \subseteq  P_h$,
   for any ideals $A$ and $B$ in $N$.
\end{defn}

\begin{exa} Let $N$ be the near-ring $(\mathbb{Z}_0[x],+,o )$ the set of all polynomial with integer coefficients
  and zero constant term, where $+$ is the usual addition and $o$ is the composition of functions.
  Let $G = (\mathbb{N}^*,*) $ the set of natural number without zero and $*$ is usual multiplication.
  $N$ is $G$-graded near-ring where $(\mathbb{Z}_0[x])_n\ = \ \mathbb{Z}X^n.$  (See [4]).
Consider the ideal $P$, the set of all polynomials in
$\mathbb{Z}_0[x] $ with even coefficient. Let $A$ and $B$ be two
ideals in $\mathbb{Z}_0[x]$  such that $A_gB_h \subseteq P_{gh}$
where $g, h \in N^* .$ Suppose that  $A_g \nsubseteq P_g $ and\ $B_h
\nsubseteq P_h$. Then there exists $a \in A_g \ $ and $\ b \in B_h$
such that neither $a$ nor $b$ has even coefficient. So $ab$ does not
have even coefficient, contradiction.
 So $A_g \subseteq  P_g\ $ or $\ B_h \subseteq P_h $ and hence $P$ is graded prime
\end{exa}
More general examples of graded prime ideals are the maximal ideal
in near-ring with multiplicative identity and the intersection of
totally ordered graded prime ideals by inclusion.

\begin{prop}
Let $N$ be a graded near-ring with multiplicative identity.
  If $P$ is a proper graded ideal does not properly contained in any proper ideal, then $P$ is  graded prime.
\end{prop}
\begin{proof}
Let $A$ and $B$ be ideals such that $A_gB_h \subseteq\ P_{gh} $.
Suppose that neither $ A_g\ \subseteq\ P_g\ $ nor
 $ B_h\ \subseteq\ P_h $, then there exists $a \in\ A_g ,\ b \in B_h\ $ such\ that\  $a \notin P_g ,\ b \notin P_h$.
  P is subset of both $P\ +\ <a>\ $ and\ $ P\ +\ <b>$. So, by assumption, $N\ =\ P\ +\ <a>\ =\ P\ +\ <b> . $
  Thus, $\ 1 \ =\  p\ +\ an\ $ and $\ 1\ =\ q\ +\ bm\ $ where $\ n,\ m \ \in\ N\ $ and\ $p,\ q\ \in\ P .\ $
  But $P$\ is\ an\ ideal, so $\ a (q\ +\ bm)\ -\ abm\ \in\ P  $ and\ since $\ 1\ =\ (q\ +\ bm),$
  then $\ (q\ +\ bm) a\ -\ abm\ \in\ P.\ $Thus, $\ q\ +\ bma\ -\ abm\ \in\ P.\ $
  Therefore $,\ bma\ \in\ P$ since $\ qa,\ abm\ \in\ P.\ $
  But, $ 1=\ 1\ *\ 1\ =\ (p\ +\ an)(q\ +\ bm)\ =\ p(q\ +\ bm)\ +\ an(q\ +\ bm)\ =\ p(q\ +\ bm)\ +\ (q\ +\ bm)an ,$
   and since $\ q\ +\ bm\ =\ 1.\ $ Then,\ $1\ =\ p(q\ +\ bm)\ +\ qan\ +\ bman\ \in\ P.\ $ Hence, $\ N\ =\ P,\ $
   contradiction.\ So,  $\ A_g\ \subseteq\ P_g  \ $ or\ $B_h\ \subseteq\ P_h.$ \ Hence, $P$ is\ graded \ prime\ ideal.
\end{proof}
\begin{prop}
Let $N$ be a $G$-graded near-ring.  Let $A$ be a totally ordered
set.
 Let $(P_a)_{a \in A}$ be a family of graded Prime ideals with $P_a \subseteq P_b$ for
 any $a,\ b \in A$ with $a \leqslant b$. Then $P = \bigcap\limits_{a \in A}P_a$ is a graded prime ideal.
\end{prop}
\begin{proof}
Let $I$ and $J$ be ideals of $N$ with $I_gJ_h \subseteq
(\bigcap\limits_{a \in A} P_a)_{gh}$.
  Then $\forall a \in A$, we have $ I_gJ_h \subseteq (P_a)_{gh}$. If  $\exists a \in A $
  such that $I_g \nsubseteq (P_a)_g$, then $J_h \subseteq (P_a)_h $. Hence, $\forall\ b \ \geqslant\  a$,
  we have $ J_h\ \subseteq\ (P_b)_h$. If $\exists\ c\ <\ a$ such that $J_h \ \nsubseteq\ (P_c)_h$,
   then $I_g\ \subseteq\ (P_c)_g$. So, $I_g \ \subseteq\ (P_a)_g$. A contradiction.
   Hence, $\forall a \in A $ we have $J_h\ \subseteq \ (P_a)$, therefore, $J_h\ \subseteq\ \bigcap\limits_{a \in A} (P_a)_h$.
\end{proof}
Proposition 2.4,
 states that if $(P_a)_{a \in A}$ is a family of
graded prime ideals with
   $P_a \subseteq P_b$ for $a,b \in A$ with $a \leqslant b$, where $A$ is totally ordered set.
    Then $P =\ \bigcap\limits_{a \in A}P_a$ is  graded prime. In general, it is not necessary
    that the intersection is graded prime   if we drop  the condition  ($A$ is totally ordered set
    and $P_a \subseteq P_b$ for $a,\ b \in A$ with $a \leqslant b$). For example,  consider the near-ring
    ($\mathbb{Z}_{6},+,*)\  $ with\ $(G={0,1},+)\ $ where\ $+$ \ defined\ as\ $0+0=0,\ 0+1=1,\ 1+0=1\ $
    and\ $1+1=1$. Then $N$ is $G$-graded near-ring defined by $N_0=\mathbb{Z}_{6},\ N_1\ = \{0\}$.
Note that the intersection of $P_1=\{0,\ 2,\ 4\}$ and $ P_2=\{0,\
3\}\ $ is not graded prime ideal, although $P_1 $ and $P_2$ are.
However, the next two theorems give some properties of any
intersection of graded prime ideals.
\\
\\
We denoted to the $I_g I_g\ ...\ I_g$ n-times by $(I_g)^n. $

\begin{thm} Let $N$ be a  graded near-ring. Let $P$ be an intersection of graded prime ideals.
   For any ideal J with $(J_g)^n  \subseteq P_{g^n} $, for some n $\in \mathbb{N}$ we have $J_g  \subseteq P_g $.
\end{thm}
\begin{proof}
Let $N$ be a graded near-ring and let $P_{\alpha's}$ be a set of
graded prime ideals in $N$.
    Let $P$ be the intersection of $P_{\alpha's}$. Let $I$ be an ideal such that $(I_g)^n \subseteq P_{g^n}$.
     Then $(I_g)^n $ is subset of each $(P_{\alpha})_{g^n} $. Since each $P_{\alpha} $ is graded prime,
     we have $I_g$  is subset of each $(P_{\alpha})_g$ or $I_{g^{n-1}}$  is subset of each $(P_{\alpha})_{g^{n-1}}$.
      If  $I_{g^{n-1}}$  is subset of each $(P_{\alpha})_{g^{n-1}}$, then we have $I_g$  is subset
      of each $(P_{\alpha})_g$ or $I_{g^{n-2}}$  is subset of each $(P_{\alpha})_{g^{n-2}}$. Consequently,
      we guarantee that   $I_g$  is subset of each $(P_{\alpha})_g$. Hence  $I_g$  is subset of
      their intersection which means $I_g$  is subset of $P_g$.
\end{proof}

\begin{thm} Let $N$ be a  graded near-ring. Let $P$ be an intersection of graded prime ideals.
       For any ideal $J$ with $J^n  \subseteq P $ for some $n \in \mathbb{N} $, we have $J  \subseteq P $.
\end{thm}
\begin{proof}
Let $P$ be the intersection of graded prime ideals and let $J$ be an
ideal with $J^n \ \subseteq\ P $.
    We have $(J_g )^n \ \subseteq \ [J^n \cap N_{g^n}] \ \subseteq\ P_{g^n} $. By Theorem 2.5,
    we have for each $ g \ \in \ G$,  $J_g \ \subseteq \ P_g $. Therefore, $J \ \subseteq\ P. $
\end{proof}

Note that, if $P$ is a prime ideal and P is graded ideal, then it is
graded prime ideal. (To see that, take two ideals $A$ and $B$ such
that $A_g B_h \subseteq P_{gh} $, then
  $A_g B_h \subseteq P $.  Let $C$ be the ideal generated by $ A_g$, $D$ be the ideal generated by
  $ B_h$, and $E$ be the ideal generated by $A_g B_h $. By the definition of the ideal generated by set,
   we have $CD \subseteq E \subseteq P $. Since $P$ is prime ideal, then $C \subseteq P $ or  $D \subseteq P $.
   Therefore, $A_g \subseteq P_g $ or  $B_g \subseteq P_g $).

Next Example shows that it is not necessary for graded prime ideal
to be prime ideal.

\begin{exa}
Let $N$ be $\mathbb{Z}[i]$ with usual addition and multiplication.
$N$ is $Z_2$-graded with  $N_0= \mathbb{Z} $ and   $N_1 =\mathbb{Z}
i $. Since $<1+i><1-i> \subseteq $
 $2N$  with neither $(1+i)$ nor $(1-i$)
belongs to $2N$,  $P = 2N$ is not prime. However, $P$ is graded
prime.

\end{exa}

Next Theorem gives an equivalent conditions for a graded ideal to be
graded prime ideal.
\begin{thm}
Let $N$ be a near-ring and $P$ be an ideal of $N$. Then the
following are equivalent:
 \begin{enumerate}
     \item For any two homogeneous elements $i$ and $j$ with $i \notin P_g$ and $ j \notin P_h$ then $<i>_g\ <j>_h\ \nsubseteq\ P_{gh}$;
     \item  For all ideals $I, J$ with $P_g \subset I_g  $ and $P_h \subset J_h $ then $I_gJ_h \nsubseteq  P_{gh}$;
     \item For all ideals $I, J$ with  $I_g \nsubseteq P_g  $ and $J_h \nsubseteq P_h $ then $I_gJ_h \nsubseteq  P_{gh}$;
     \item $P$ is graded prime ideal.
 \end{enumerate}
\end{thm}

\begin{proof}
$(1) \Rightarrow (2) :\ $ If $P_g \subset I_g $ and $P_h \subset J_h
$.
 Take $i \in I_g$  and $j \in J_h $ with $i \notin P_g $ and $j \notin P_h$.
 Then $<i>_g<j>_h \nsubseteq P_{gh} $. Hence $I_gJ_h$ $\nsubseteq P_{gh}$.
 $(2) \Rightarrow (3):\ $ If $I_g \nsubseteq P_g $ and $J_h \nsubseteq  P_h$.
 Take $i \in I_g$  and $j \in J_h $ with $i \notin P_g $ and $j \notin P_h $.
 Then $(<i>_g\ +\ P_g)$ $(<j>_h\ +\ P_h)$ $\nsubseteq\  P_{gh}$.
 Therefore, $\exists  i_1\ \in\ <i>_g ,\ j_1\ \in\ <j>_h$,  $p_1\ \in\ P_g$ and
 $p_2\ \in\ P_h $ such that: $(i_1\ +\ p_1)(j_1\ +\ p_2)\ \notin\ P_{gh}$.
 So, $i_1 (j_1\ +\ p_2)\ -\ i_1j_1 \ +\ i_1j_1\ +\ p_1(j_1\ +\ p_2)\ \nsubseteq\ P_{gh}$.
 So, since $i_1 (j_1\ +\ P_2)\ -\ i_1j_1\ \in\ P_{gh}$ and $i_1 j_1\ \notin\ P_g$  we get $I_gJ_h\ \nsubseteq\
 P_{gh}$.
\\
(3) $ \Rightarrow$ (4) :\ Follows directly from the definition of
graded prime ideal.
\\
 (4) $ \Rightarrow$ (1) :\ If  $<i>_g\ <j>_h\
\subseteq\ P_{gh}$,
  then $<i>_g\ \subseteq\ P_g $ or  $<j>_h\ \subseteq\ P_h $. Therefore, $i  \in P_g$ or $j \in P_h$.

\end{proof}

Next proposition gives another equivalent condition that guarantees
a graded ideal to be graded prime.

\begin{prop}
Let $P$ be a graded prime ideal of a graded near-ring $N$. Then the
following are equivalent:
\begin{enumerate}
 \item \textnormal{For $a \in N_h\ $ and\ $b, c \in N_g,\ $
 with\ $ a (< b >_g+< c >_g)\ \subseteq\ P_{hg}$,\ we\ have\ a $\in P_h,$\ or\ $b$\ and $ c\in  P_g$}.
\item  \textnormal{For\ $x \in N\ $ but\ $x_g \notin P_g,\ $ we\ have\ $(P_{hg} : < x >_g +< y >_g)_h = P_h$ for\ any\ $y\in N_g.$}

\item \textnormal{P is graded  prime}.
\end{enumerate}
\end{prop}

\begin{proof}
$ (1)  \Rightarrow  (2) :\ $ Let\ $t \in N_h\ $ and\ $t \in (P_{hg}
: < x >_g +< y >_g)\ $ for\ any\ $y$ and $x \in N_g $ but\ $x\
\notin\ P_g .\ $ Therefore,\ $t (< x >_ g\ +\ < y >_g)\ \subseteq\
P_{hg}.\ $ Thus, by\ hypothesis, \ $t \in  P_h\ $.

$ (2) \Rightarrow (3) :\ $ Let\ $A$ and $B$\  be\ graded \ ideals\
of\ $N$\ such\ that\ $A_h B_g \subseteq P_{hg} .\ $ Suppose\ that $\
A_h\ \nsubseteq\ P_h\ $ and\ $B_g\ \nsubseteq\ P_g $,\ then there
exists $ b \in B_g\ $ with\ $b \notin P_g .$ \ Now\ we\ claim\ that\
$A_h B_g =0.\ $ Let\ $b_1 \in B_g ,\ $ then $A_h (< b >_g +< b_1
>_g) \subseteq P_{hg},$\ and then $ A_h \subseteq (P_{hg}:< b >_h +<
b_1
>_h)= P_h,$\ which is a contradiction.  Hence,\ $P$\ is\ graded\
prime.

$ (3) \Rightarrow\ (1) :  $ If $ a (< b >_g+< c >_g)\ \subseteq\
P_{hg}$, then
 $ <a>_h (< b >_g+< c >_g)\ \subseteq\ P_{hg}$. Hence, $<a>_h \subseteq P_h$ or
  $(< b >_g+< c >_g) \subseteq P_h$ since P is graded prime. Therefore, $a \in P_h$  and
   $ b,c \in P_g$, since $a \in N_h$  and  $ b,c \in N_g$.

\end{proof}

Next, we use the properties of graded prime ideals to study some
properties of some special graded near-rings such as graded quotient
near-rings, graded near-rings with zero ideal is graded prime and
the product of two graded near-rings.
\begin{thm}
Let $G$ be a multiplicatively monoid with identity element and let
$N$ and $M$ be two $G$-graded near-rings.
 Let $\Phi $ be a surjective homomorphism such that $\Phi (I_g) = (\Phi(I))_g $ for any ideal $I \in N$ and any $g \in G$. Then:
\begin{enumerate}
    \item[(i).] The pre-image of graded prime ideal is graded prime.
    \item[(ii).] The image of graded prime ideal which contains the kernal of $\Phi $ is  graded prime.
\end{enumerate}
\end{thm}
\begin{proof}
$(i)$ Suppose that $A_g B_h\ \subseteq\ (\Phi^{-1}(P))_{gh} $, where
$A$ and $B$ are ideals of $N$,
 and $P$ is graded prime ideal of $M$. Then we have  $(\Phi(A))_g  (\Phi(B))_h \subseteq P_{gh}$.
  Since $P$ is graded prime ideal, $(\Phi(A))_g \subseteq P_g $ or $(\Phi(B))_h \subseteq P_h $.
  Then, $A_g \subseteq (\Phi^{-1}(P))_g$ or $B_h \subseteq (\Phi^{-1}(P))_h$. Hence, $\Phi^{-1}(P)  $ is graded prime ideal.

$(ii)$ Suppose that $A_g B_h\ \subseteq\ (\Phi(P))_{gh} $, where $A$
and $B$ are ideals of $M$ and $P$ is graded prime ideal of $N$. Then
we have $(\Phi^{-1}(A))_g  (\Phi^{-1}(B))_h \subseteq (P_{gh} +\ Ker
\Phi) \subseteq P$. However, $(\Phi^{-1}(A))_g  (\Phi^{-1}(B))_h
\subseteq N_{gh}$ hence $(\Phi^{-1}(A))_g  (\Phi^{-1}(B))_h
\subseteq P_{gh}$. Since $P$ is graded prime ideal,
$(\Phi^{-1}(A))_g \subseteq P_g $ or $(\Phi^{-1}(B))_h \subseteq P_h
$. Then,
 $A_g \subseteq (\Phi(P))_g$ or $B_h \subseteq (\Phi(P))_h$. Hence, $\Phi^{-1}(P)  $ is graded prime ideal.

\end{proof}

\begin{lem} Let $N$ be a graded near-ring. If $h: N \rightarrow  \Bar{N}:= N/I $ is
 a homomorphism with $h(J_g) = h(J)_g,$ then $h^{-1}(J_g) =h^{-1}(J)_g $.
\end{lem}

\begin{proof}

$h^{-1}(J)_g =h^{-1}(J)\cap N_g= h^{-1}(h(h^{-1}(J)\cap N_g)) =
h^{-1}(h(h^{-1}(J))\cap (N/I)_g) = h^{-1}(J\cap (N/I)_g) =
h^{-1}(J_g).$
\end{proof}
\begin{thm}
Let $N$ be a  graded near-ring, $P$ be a graded ideal, and $Q$ be a
graded ideal of $N$ with $ Q \subseteq P$. Let $\pi : N \rightarrow
\Bar{N}:= N/Q $ be the canonical epimorphism. Then $P$ is graded
prime if and only if $\pi(P) $ is graded prime.

\end{thm}

\begin{proof}
Let $P$ be a graded prime ideal of $N$. Let $J$ and $I$ be two
ideals of $\Bar{N}$ with $J_g I_h \subseteq (\pi (P))_{gh} $.  Let
$j_g:= (\pi^{-1}(J))_g$ and $i_h:= (\pi^{-1}(I))_h$. Then we have,
by Lemma 2.11, $j_g:= (\pi^{-1}(J_g))$ and
  $i_h:= (\pi^{-1}(I_h))$. So, $j_gi_h = (\pi^{-1}(J_g))(\pi^{-1}(I_h))$ $\subseteq (\pi^{-1}(J_gI_h)) $
  $\subseteq \pi^{-1}((\pi(P))_{gh})$ $= (\pi^{-1}(\pi(P)))_{gh} $ $ = (P+Q)_{gh} = P_{gh}.$
and hence $j_g \subseteq P_g  $ or  $i_h \subseteq P_h  $.
Therefore, $J_g = (\pi(\pi^{-1}(J)))_g$ $= \pi(\pi^{-1}(J)_g) $ $ =
\pi(j_g) \subseteq \pi(P_g ) =  (\pi(P)_g  $\ or\ $I_h  \subseteq
(\pi(P ))_h.$
 Thus, $\pi(P)  $ is graded prime.

 For the other direction, let $ \pi(P)$\ be prime ideal of\ $\Bar{N}.
$ If  $j_g i_h \subseteq P_{gh}, $\ then\ $(\pi(j))_g (\pi(i))_h =
\pi(j_gi_h) \subseteq (\pi (P_{gh})) =(\pi (P))_{gh} .$ And hence\
$(\pi(j))_g\ \subseteq\ (\pi(P))_g$\ or\ $(\pi(i))_h\ \subseteq\
(\pi(P))_h$. Therefore,
 $J_g \subseteq (j + Q)_g = (\pi^{-1}(\pi(j)))_g =\pi^{-1}((\pi(j))_g) \subseteq \pi^{-1}((\pi(P))_g) = (\pi^{-1}(\pi(P)))_g = (P+Q)_g = P_g$ or
 $I_h \subseteq P_h.$

\end{proof}

Next proposition gives a property that could be useful to determine
weather some ideals are graded prime or not. Furthermore, the next
proposition, gives  characteristics of the quotient
 graded near-ring $N/P $, whenever $P$ is graded prime ideal in the graded near-ring $N$.
\begin{prop}
Let $N$ be a graded near-ring and $A$, $B$ be ideals.
 $P$ is graded prime ideal if and only if  $\Bar{A}_g\Bar{B}_h  \neq 0 $ if both
  $\Bar{A}_g$  $\Bar{B}_h  \neq 0 $     in $N/P$.
\end{prop}
\begin{proof}
Let $P$ be a graded prime ideal and $A$, $B$ be  any ideals such
that $\Bar{A}_g \neq 0$ and $\Bar{B}_h \neq 0$ in $N/P $ for some
$g,\ h \in G $.
 Hence, neither $A_g \subseteq P_g$ nor $B_h \subseteq P_h$.
 To show that $\Bar{A}_g \Bar{B}_h \neq 0 $. Suppose that $\Bar{A}_g \Bar{B}_h\ =\ 0 $.
 Hence, $A_g B_h \subseteq P_{gh} $, contradicts Theorem 2.8.

Conversely, Let $A_g B_h \subseteq P_{gh}$, then, $\Bar{A}_g
\Bar{B}_h = 0$ in $N/P $.
 By assumption,  $\Bar{A}_g = 0  $ or $\Bar{B}_h = 0  $  in $N/P $.  Therefore,
   $A_g \subseteq P_g $ or $B_h \subseteq P_h$. Hence, $P$ is graded prime.
\end{proof}

We know that in some near-rings, for instance integral near-rings,
 the zero ideal is prime ideal and we know that these near-rings are called prime
  near-rings. However, we can note that in some graded near-rings,  the Zero ideal
   is graded prime. For example, $I = \ \{0\}$ is graded prime in the graded near-ring
   which defined in Example 2.2. In Next, we study  few interesting properties of such near-rings.
\\
If a graded near-ring $N$ is simple, then there is no proper ideal
unless $\{0\} $.
 Hence,  by definition of graded prime ideal,  $\{0\} $ is graded prime ideal or $N$ is
  a zero  near-ring.  More general results is discussed in the Proposition 2.15.
   Before that we need the following  lemma.

\begin{lem}
Let $N$ be a graded near-ring and let $I$ be a graded ideal of $N$.
Zero is graded prime ideal in the graded near-ring $N/I$ if and only
if the ideal $I$ is a graded prime ideal of $N$.
\end{lem}

\begin{proof} Follows by taking $I = P$ in Theorem 2.12.
\end{proof}

\begin{prop} Let N be a graded near-ring and $I$ be a maximal ideal.
  If $I$ is graded ideal, then either $I$ is graded prime or $N^2 \subseteq I $.
\end{prop}

\begin{proof} Let $N$ be a graded near-ring and  $I$ be a maximal ideal which can be graded.
$N/I $ is simple near-ring. Hence, $ \{0\}$ is graded prime in $N/I
$  or $N/I$ is a zero near-ring, which implies that either $I$ is
graded prime ideal by Lemma 2.14, or $N^2 \subseteq I$.
\end{proof}

\begin{cor} Let $N$ be a graded near-ring with unity and let $I$ be a graded ideal.
      If $I$ is a maximal ideal in $N$, then $I$ is a  graded prime ideal.
\end{cor}

\begin{proof}
Let $I$ be a maximal ideal. If $1$ is a unity of $N$, then for any
$n \in N$, $ n\ =\ n.1 \in N^2 $. Hence, $N^2 = N \nsubseteq I $. By
previous theorem,  $I$ is graded prime.
\end{proof}

\begin{prop} Let $G$ be a multiplicatively monoid with identity element and let $N$ and
   $M$ be two $G$-graded near rings. Let $\Phi $ be a surjective homomorphism from $N$ into $M$,
   such that $\Phi (I_g) = (\Phi(I))_g $ for any ideal $I \in N$ and any $g \in G$. If $\{0\}$ is
   graded prime ideal in $M$, then the kernel is graded prime in $N$.
\end{prop}

\begin{proof}
Follows by Theorem 2.10.
\end{proof}
Note that in the previous proposition, it is not necessary  that $\{
0\}  $
 is graded prime ideal in $N$ if $\{0\}$ is graded prime in $M$.  For example, if
   $\ (G=\{0,1\},+)$\ where\ +\ defined\ as\ $0+0=0,\ 0+1=1,\ 1+0=1\ $ and\ $1+1=1,$\
    then  $M = \mathbb{Z}_2$ and $N = \mathbb{Z}_8$ both are $G$-graded near rings, where,
     $M_0= \mathbb{Z}_2,\ M_1\ = \{0\}$ and $N_0= \mathbb{Z}_8,\ N_1\ = \{0\}$.
     $\Phi:\mathbb{Z}_8 \rightarrow \mathbb{Z}_2,\ \Phi(x)= x $ satisfies the mentioned
     conditions in the previous proposition. However, $\{0\} $ is not graded prime in $\mathbb{Z}_8$
          although it is graded prime ideal in $\mathbb{Z}_2$.   We can also note that  it is
          not necessary  that $\{ 0\}  $ is graded prime ideal in $M$ if $\{0\}$ is graded prime in $N$.
           For example, Let G be defined as Example 2.2, and let N = $\mathbb{Z} $ and
           $M = \mathbb{Z}_8 $ with   $N_0= \mathbb{Z}_,\ N_1\ = \{0\}$ and $M_0= \mathbb{Z}_8,\ M_1\ = \{0\}$.
 Then we have $\{0\} $ is graded prime ideal in $\mathbb{Z}$, but it is not graded prime in $\mathbb{Z}_8$.
  Although, $\Phi (x) = x $, is surjective homomorphism from $\mathbb{Z}$ into
  $\mathbb{Z}_8$.\\

  Let $G$ be a multiplicatively monoid with identity. Let  $(N_1,+_1,*_1)$ and $(N_2,+_2,*_2)$ both
   be $G$-graded near-rings. Define $ N  =  N_1 \times N_2 = \{(n_1,n_2):\ n_1 \in N_1,\ n_2 \in N_2\}$.
   The Direct Product is the set $N$ paired with the operations  addition $+$ and multiplication $*$
   defined as $(n_1^\prime\ ,\ n_2^\prime)\ +\ (n_1^\prime{} ^\prime\ ,\ n_2^\prime{} ^\prime)$ $= $
   $(n_1^\prime +_1 n_1^\prime{} ^\prime\ ,\ n_2^\prime +_2  n_2^\prime {}^\prime)$ and
   $(n_1^\prime\ ,\ n_2^\prime)\ *\ (n_1^\prime{} ^\prime\ ,\ n_2^\prime{} ^\prime)$ $= $
   $(n_1^\prime *_1 n_1^\prime{} ^\prime\ ,\ n_2^\prime *_2  n_2^\prime{} ^\prime)$, for
   each $(n_1^\prime\ ,\ n_2^\prime),\ (n_1^\prime {}^\prime\ ,\ n_2^\prime{} ^\prime) \in N$.
   $N$ is graded near-ring by $N_g = (N_1)_g \times (N_2)_g $.
  Since it is known that if $ (N_1)_g $ and $(N_2)_g $ are normal subgroups of $N_1 $ and $ N_2$,
  respectively, then $N_g = (N_1)_g \times (N_2)_g $ is normal subgroup of $N$.
  Also for any $n = (n_1 ,n_2) \in N$ is written uniquely as sum of finite
  elements of $ N_g=\ (N_1)_g \times (N_2)_g  ,\ g \in G$ since $n_1$ and $n_2$
  are written uniquely as sum of finite elements of $  (N_1)_g$ and $(N_2)_g  ,\ g \in G$
  respectively. Furthermore, $ N_g N_h =  [(N_1)_g \times (N_2)_g] [(N_1)h \times (N_2)_h]= ((N_1)_g(N_1)_h) \times ((N_2)_g(N_2)_h)
     \subseteq (N_1)_{gh} \times  (N_2)_{gh}= N_{gh}$
\\

  Note that If $P_1$ and $P_2$ are graded prime ideals in $N_1$ and $N_2 $, respectively, then $P_1 \times P_2 $ need not be graded prime ideal of $N_1 \times
  N_2$. For example, $\{0\} $ is graded prime in $N = \mathbb{Z}_2 $ where
  $G  = \mathbb{Z}_2$ with $ N_0=\{0\}$ and $N_1 = \mathbb{Z}_2 $, while $\{0\} \times \{0\} $ is not
   graded prime in $\mathbb{Z}_2 \times \mathbb{Z}_2$ since $(\mathbb{Z}_2 \times \{0\})_1$ $(\{0\} \times
   \mathbb{Z}_2)_1$ $\subseteq$ $(\{0\} \times \{0\})_1 $  although neither $(\mathbb{Z}_2 \times \{0\})_1$
     $\subseteq$ $(\{0\} \times \{0\})_1 $ nor $(\{0\} \times \mathbb{Z}_2)_1$ $\subseteq$ $(\{0\} \times \{0\})_1 $.
\\

\begin{thm}
Let $N$ and $M$ be any two $G$-graded near-rings with identity and
$P$ be a graded proper ideal of $N$. Then $P$ is graded prime if and
only if $P \times M$ is a graded  prime ideal of $N \times M$
\end{thm}

\begin{proof}
Let $P$ be a graded  prime ideal of $N$ and let $(A \times B)$  and
$(C \times D)$ be ideals of $N \times M$ such that $(A \times B)_g(C
\times D)_h \subseteq (P \times M)_{gh} $.\ Then\ $(A_gC_h \times
B_gD_h) \subseteq (P_{gh} \times M_{gh}).$
 So,\ $A_gC_h \subseteq P_{gh}.$  Which\ implies $\
A_g  \subseteq P_g\ $ or\ $C_h \subseteq P_h. \ $ Hence, $(A \times
B)_g \subseteq (P \times M)_g$ or
 $(C \times D)_h \subseteq (P \times M)_{h}$. Hence $P \times M$ is graded prime ideal.

Conversely,\ suppose\ that $\ (P \times M)\ $ is\ a\ graded\  prime\
ideal\ of\ $N \times M \ $ and\ let $I$ and $J$ be\ ideals\ of $N$\
such\ that\ $I_gJ_h \subseteq P_{gh} $.
  Then\ $(I \times M)_g(J \times M)_h \subseteq (P \times M)_{gh}.$  By\
assumption,\ we\ have\ $(I \times M)_g \subseteq (P \times M)_g\ $
or\ $(J \times M)_h \subseteq (P \times M)_h.$\ So,\ $I_g  \subseteq
P_g\ $ or\ $J_g \subseteq P_g.$ Hence, $P$ is graded prime ideal.
\end{proof}

\begin{cor}
Let $N$ and $M$ be two $G$-graded near-rings with identity. If every
proper ideal of $N$ and $M$ is a product of graded  prime ideals,
then every proper ideal of $N \times M$ is a product of graded prime
ideals.
\end{cor}

\begin{proof}
Let $I$ be a proper ideals of $N$ and $J$ be a proper ideals of $M$,
Such that $I = A_1\ ...\ A_n\ $ and $\ J\ = \ B_1\ ...\ B_m\ $ where
\ each
 $\ A_i\ $ and $\ B_j$ is graded  prime. If the
proper ideal is of the form $I \times M$, then  $I \times M =
A_1...A_n \times M\ $ can\ be \ written \ as $ \ (A_1\ \times\ N_2)\
...\ (A_n\ \times\ M)$ which is by Theorem 2.18 a product of graded
prime ideals. Similarly, If the proper ideal is of the form $N
\times J$, then it is a product of graded prime ideals. If the
proper ideals is of the form $I \times J$, then it can be written as
$A_1\ ...\ A_n\ \times\ B_1\ ...\ B_m\ = \ (A_1\ ...\ A_n\ \times\
M) \ (N\ \times\ B_1\ ...\ B_m)\ =\ (A_1 \times M)\ ...\ (A_n \times
M)\ (N \times B_1)\ ...\ (N \times B_m).$ Which is a product of
graded  prime ideals.
\end{proof}

\begin{thm}
Let $N$ and $M$ be  two $G$-graded near-rings with unity. Then a
graded ideal $P$ of $N \times M$ with $ (N \times M)_g = N_g \times
M_g,\ \forall g \in G  $
 is graded  prime if and only if it has one of the following two
 forms:
\begin{enumerate}
    \item  $(I \times M)$, \textnormal{where $I$ is a graded  prime ideal of $N$}.

    \item (N $\times$ J), \textnormal{where $J$ is a graded  prime ideal of $M$}.

\end{enumerate}
\end{thm}

\begin{proof} Let $P$ be a proper graded  ideal of $N \times M.$ Then $P$ has one of the following three forms:
  $(I \times M)$  where $I$ is proper ideal of $N$,  $(N \times J)$, where   $J$ is proper ideal of $M$,
   or $I \times J$, where  $I_g\ \neq\ N_g\ $  and $\ J_h\ \neq\ M_h $ for some $g$ and $h$ belongs to $G$.
   \\
   If $P$ is of the form $(I \times M)$ or of the form $(N \times J),$
then, by Theorem 2.18, $P$ is graded  prime ideal if and only if
both $I$ and $J$ are graded prime. Let $P = (I \times J)$ be a
graded prime ideal  with  $I_g\ \neq\ N_g\ $  and
 $\ J_h\ \neq\ M_h $, for some $g$ and $h$ belongs to $G$.
 Suppose $ a \ \in\ I_{gh} $.\ Then\ $(< a >_{gh}\ \times\ \{0\})\ \subseteq\ P_{gh}$ \ This\ implies\
that\ either\ $(< a >\ \times\ M)_g\ \subseteq\ P_g\ $ or\ $(N\
\times\ \{0\})_h\ \subseteq\ P_h$.\ If\ $(< a >\ \times\ M)_g\
\subseteq\ P_g $, \ then\
 $M_g\ =\ J_g$\ and\ if\ $(N\ \times\ \{0\})_h\ \subseteq\ P_h$,\ then\ $N_h\ =\ I_h$.\
 Contradiction.\  Hence $I$ $\times  $ $J$ can not be graded  prime ideal if both $I$ and $J$ are proper ideals.
\end{proof}

Note that, from  previous theorem, we can not find a graded
near-ring in which every
 proper graded ideal is graded prime if the $G$-graded near-ring $N$ with unity and  of the
  form $N_1 \times N_2$, where ($N_1 \times N_2)_g = (N_1)_g \times (N_2)_g$ unless one
   of $N_1$ or $ N_2$ is the Zero near-ring. More general,  we can not find graded near-ring
    with every proper graded ideal is graded prime regardless it has unity or not.  Since the
     ideal $\{0\} \times \{0\}  $ can not be graded prime ideal if both $ N_1$ and $N_2$ are non trivial near-rings.

\begin{prop}
Let $N$ and $M$ be  two non trivial $G$-graded near-rings. Then
$\{0\} \times \{0\}$
 can not be graded prime ideal of $N \times M$ with $ (N \times M)_g = N_g \times M_g,\ \forall g \in G  $.
\end{prop}

\begin{proof}
Suppose that $\{0\}\ \times\ \{0\}$ is graded prime ideal for some
$G$-graded near-ring
 $ N\ \times\ M$ with  $ (N\ \times\ M)_g\ =\ N_g\ \times\ M_g,\ \forall g\ \in\ G  $.
 Take $g$ and $h$ belongs to $G$ such that $N_g\ \neq\ \{0\}  $ and $M_h\ \neq\ \{0\} $.
 Since $(\{0\}\ \times\ N)_g (M\ \times\ \{0\})_h\ \subseteq\ (\{0\}\ \times\ \{0\})_{gh}$
 for any $g$ and $h$ belongs to $G$, then $(\{0\}\ \times\ N)_g\ \subseteq\ (\{0\}\ \times\ \{0\})_g$ or
  $ (M\ \times\ \{0\})_h\ \subseteq\ (\{0\}\ \times\ \{0\})_{h}$.
Therefore, $N_g\ =\ \{0\}$ or $M_h\ =\ \{0\}.$ Contradiction and
hence neither $(\{0\}\ \times\ N)_g\ \subseteq\ (\{0\}\ \times\
\{0\})_g$ nor $ (M\ \times\ \{0\})_h \ \subseteq\ (\{0\}\ \times\
\{0\})_{h}$, which implies $\{0\}  \times \{0\}$  can not be graded
prime ideal.
\end{proof}

\begin{cor}
If every graded ideal is graded prime in some $G$-graded near-ring
  $N \times  M$ with  $ (N \times M)_g = N_g \times M_g,\ \forall g \in G  $, then
  every graded ideal is graded prime of one of $N$ and $M$ while the other is the trivial near-ring.
\end{cor}

\begin{proof}
Suppose that every graded ideal is graded prime in some  $N \times
M$ with neither
  $N$ nor $M$ is Zero near-ring. Hence, $\{0 \} \ \times\ \{0\}$ is graded prime which contradicts previous proposition.
\end{proof}

\bibliographystyle{amsplain}

\end{document}